\documentclass[11pt]{amsart}

\usepackage{amsmath}
\usepackage{amsfonts}
\usepackage{amssymb}
\usepackage{amsthm}
\usepackage{hyperref}
\usepackage{enumerate}
\usepackage{cleveref}
\usepackage{mathrsfs}
\usepackage[top=0.9in, bottom=1.15in, left=1.1in, right=1.1in]{geometry}
 \usepackage{hyperref}
\hypersetup{colorlinks=true,linkcolor=blue,citecolor=magenta}
\newtheorem{theorem}{Theorem}

\newtheorem{corollary}[theorem]{Corollary}

\Crefname{conjecture}{Conjecture}{Conjectures}

\theoremstyle{plain}

\theoremstyle{plain}

\newtheorem*{thmnonum}{Theorem}
\newtheorem*{FDB}{Fa\`a di Bruno's formula}
\newtheorem*{MD}{MacMahon's decomposition}

\theoremstyle{definition}
\newtheorem{definition}[theorem]{Definition}
\theoremstyle{remark}

\newcommand{\calP}{\mathcal{P}}

\newcommand{\calC}{\mathcal{C}}

\allowdisplaybreaks

\numberwithin{equation}{section}
\numberwithin{theorem}{section}
\author{Robert Schneider and Andrew V. Sills}

\address{Department of Mathematics\newline
University of Georgia\newline
Athens, Georgia 30602}
\email{robert.schneider@uga.edu}

\address{Department of Mathematical Sciences\newline
Georgia Southern University\newline
Statesboro, Georgia 30458}
\email{asills@georgiasouthern.edu}

\title[Analysis and combinatorics of partition zeta functions]{Analysis and combinatorics of partition zeta functions}
\begin{document}
\begin{abstract}
We examine ``partition zeta functions'' analogous to the Riemann zeta function but 
summed over subsets of integer partitions. We prove an explicit formula for a family of 
partition zeta functions already shown to have nice properties --- those summed over 
partitions of fixed length --- which yields complete information about analytic 
continuation, poles and trivial roots of the zeta functions in the family. Then we present a 
combinatorial proof of the explicit formula, which shows it to be a zeta function analog of 
MacMahon's partial fraction decomposition of the generating function for partitions of 
fixed length. 
\end{abstract}
\maketitle

\centerline{Dedicated to Bruce Berndt on the occasion of his $80${th} birthday}
\  

\section{Introduction: partition zeta functions}
Here we study an interesting class of objects dwelling at the intersection of partition theory and the theory of $\text{L}$-functions. Let $\mathcal P$ denote the set of integer partitions (see e.g. \cite{Andrews, Berndt}), with $\mathcal P_{S}$ being partitions whose parts all belong to a subset $S\subseteq \mathbb N$ of natural numbers. 
Let $\lambda=(\lambda_1, \lambda_2,\lambda_3,...,\lambda_r)$ denote a generic partition, $
\lambda_1 \geq \lambda_2 \geq ... \geq \lambda_r \geq 1$, with $\emptyset$ the empty 
partition. 
Let $|\lambda|$ denote the sum of parts or {\it size} of $\lambda$, with $|
\emptyset|:=0$. 
Let $N(\lambda)$ denote the product of the parts or {\it norm} of the 
partition, with $N(\emptyset):=1$. 
Let $\ell(\lambda):=r$ denote the number of parts or \emph{length} of a partition, with $\ell(\emptyset):=0$. 

Then in analogy to the Riemann zeta function $\zeta(s):=\sum_{n=1}^{\infty}{n^{-s}}$, 
convergent for $\operatorname{Re}(s)>1$, 
we give the definition of a {\it partition zeta function} as studied in \cite{ORS, Robert_zeta, PhD}.
\begin{definition}\label{ch1pzf}
For a proper subset $\mathcal P'\subset \mathcal P$ 
and value $s\in \mathbb C$ for which the series converges, 
we define a {\it partition zeta function} to be the following sum over partitions in $\mathcal P'$: 
\begin{equation}\label{ch1zetadef}
\zeta_{\mathcal P'}(s):=\sum_{\lambda\in \mathcal P'} N(\lambda)^{-s}.
\end{equation}
\end{definition}

We note that $\zeta_{\mathcal P}(s)$ itself diverges\footnote{And similarly if the maximum number of $1$'s appearing in partitions in $\mathcal P'$ is unrestricted}: partitions of the shape $\lambda=(1,1,1,...,1)$ contribute infinitely many $1$'s to the sum. If $\mathcal P'$ is of the class $\mathcal P_{S}$ of partitions into elements from some subset $S \subsetneq \mathbb N$ where $1\not\in S$, then the associated zeta function has an Euler product as well:
\begin{equation}\label{eulerprod}
\zeta_{\mathcal P_{S}}(s)=\prod_{n\in S} \left(1-n^{-s}\right)^{-1},
\end{equation}
and standard techniques working with products give other variants. 

Of course, in this setting the Riemann zeta function $\zeta(s)$ represents the case $\zeta_{\mathcal P_{\mathbb P}}(s)$ (sum over partitions into prime parts). These combinatorial zeta functions (as well as related {\it partition Dirichlet series}) form a highly general class of objects, yet they share many structural laws, which specialize to well-known classical zeta function (and Dirichlet series) identities, as well as more exotic non-classical cases such as the following, proved in \cite{Robert_zeta}, for partitions into even parts, distinct parts, and parts not equal to 1, respectively: 
\begin{flalign}\label{ch1examples}
\zeta_{\mathcal P_{\text{even}}}(2)=\frac{\pi}{2},\  \  \  \  \  \  \  \  \   \zeta_{\mathcal 
P_{\text{dist}}}(2)=\frac{\operatorname{sinh} \pi}{\pi},\  \  \  \  \  \  \  \  \   \zeta_{\mathcal 
P_{\neq1}}(3)=\frac{3\pi}{\operatorname{cosh}\left(\frac{1}{2}\pi \sqrt{3}\right)}
.\end{flalign}
Since choice of subset $\mathcal P' \subsetneq \mathcal P$ as well as domain of $s\in \mathbb C$ determine a given partition zeta function, each pair of choices yields distinctive yet disparate evaluations and analytic properties. 

Then identifying interesting non-classical cases seems like something of a needle-in-a-haystack type of undertaking. One wonders: are there non-classical families of partition 
zeta functions having such nice properties that they truly parallel, for instance, Euler's 
evaluations of $\zeta(s)$ at positive even arguments: 
$$\zeta(2m)=\pi^{2m}\times \text{rational number?}$$ 
\  

\section{A nice family of partition zeta functions}\label{Sect2}
A particularly nice class of zeta functions is defined in \cite{Robert_zeta}, that generalizes $\zeta(s)$ in a more fruitful direction than just writing $\zeta(s)=\zeta_{\mathcal P_{\mathbb P}}(s)$.
\begin{definition} For $\operatorname{Re}(s)>1$, we define the family of zeta sums taken over all partitions of fixed length $k\geq 0$:
$$\zeta_{\mathcal P}(\{s\}^k):=\sum_{\ell(\lambda)=k}\frac{1}{N(\lambda)^s},$$
with $\zeta_{\mathcal P}(\{s\}^0):=N(\emptyset)^{-s}=1$.
\end{definition}

The $k=1$ case is $\zeta(s)$. For $|z|<1$ we have the obvious generating function:
\begin{equation}\label{zetagen}
\prod_{n=1}^{\infty}\left(1-z n^{-s} \right)^{-1}=\sum_{k=0}^{\infty}\zeta_{\mathcal P}(\{s\}^k)z^k.
\end{equation}

The first author proves in \cite{Robert_zeta} that for argument $s=2$ and $k\geq 1$, these partition zeta functions 
are actually rational multiples of Euler's zeta values (and of $\pi^{2k}$).
\begin{thmnonum}[Schneider]\label{zeta} For $k\geq 1$ we have 
\begin{equation*}
\zeta_{\mathcal P}(\{2\}^k) = \frac{2^{2k - 1} - 1}{2^{2k-2}}\zeta(2k).
\end{equation*}
\end{thmnonum}
We note that setting $k=0$ in Proposition \ref{zeta} suggests formally that $\zeta(0) = 
\frac{2^{-2}}{2^{-1} - 1}\zeta_{\mathcal P}(\{2\}^0) = -1/2$, which is the correct value for $
\zeta(0)$ obtained through analytic continuation. 
In \cite{ORS}, Ono, Rolen and the first author 
prove other facts about partition zeta functions, including 
special cases yielding analytic continuation (restricted, however, to the right half of the 
complex plane), and a farther-reaching follow-up to the preceding theorem.
\begin{thmnonum} [Ono--Rolen--Schneider]
For $m\geq 1, k\geq 1,$ we have
$$\zeta_{\mathcal P}(\{2m\}^k)=\pi^{2mk}\times \text{rational number}.$$ 
\end{thmnonum}
So these zeta sums over partitions of fixed length do indeed form a family like Euler's 
zeta values; but zeta values are only the beginning of the story of the zeta function, as 
Riemann displayed in a brilliant sequel to Euler's work. Then it is natural to ask about 
analytic properties of the general case $\zeta_{\mathcal P}(\{s\}^k), \operatorname{Re}
(s)>1$. Here we prove a closed formula for the partition zeta function for each $k\geq 
0$, that speaks to this question. 

Let ``$\lambda \vdash n$'' mean that $\lambda$ is a partition of size $n\geq 0$, 
and let $m_j=m_j(\lambda)\geq 0$ denote the {\it multiplicity} of $j\geq 1$ as a part of 
partition $\lambda$.
%
%
\begin{theorem}\label{Pennthm}
For $\operatorname{Re}(s)>1, k \geq 0$, we have  
$$\zeta_{\mathcal P}(\{s\}^k)=\sum_{\lambda \vdash k}
\frac{\zeta(s)^{m_1}\zeta(2s)^{m_2}\zeta(3s)^{m_3}\cdots \zeta(ks)^{m_k}}{N(\lambda)\  
m_1!\  m_2!\  m_3! \cdots m_k!},$$
where the sum on the right-hand side is taken over the partitions of size $k\geq 0$.
\end{theorem}

To prove the theorem, we need the Maclaurin expansion of the natural logarithm for $|x|<1$: 
\begin{equation}\label{Maclaurin}
-\log(1-x)=\sum_{j=1}^{\infty}\frac{x^j}{j},
\end{equation} 
as well as a partition version of {\it Fa\`{a} di Bruno's formula}. For completeness, we give a quick proof of the classical identity below.


\begin{FDB}\label{Faa}
Take $\operatorname{exp}(t):=e^t$ and $a_j, x \in \mathbb C$ such that $\sum_{j=1}^{\infty}a_j x^j$ converges. Then we have
$$\operatorname{exp}\left({\sum_{j=1}^{\infty}a_j x^j}\right)=\sum_{\lambda\in \mathcal P}x^{|\lambda|}\frac{a_1^{m_1}a_2^{m_2}a_3^{m_3}\cdots}{m_1!\  m_2!\  m_3!\cdots}=\sum_{k=0}^{\infty}x^k \sum_{\lambda \vdash k}\frac{a_1^{m_1}a_2^{m_2}a_3^{m_3}\cdots\  a_k^{m_k}}{m_1!\  m_2!\  m_3!\cdots  m_k!}.$$
 \end{FDB}

\begin{proof}[Proof of Fa\`a di Bruno's formula]
We begin with the classical {\it multinomial theorem}, written as a sum over length-$n$ partitions $\lambda$ in the set $\mathcal P_{[k]}\subset \mathcal P$ whose parts are all $\leq k$: 
\begin{equation}\label{Dmultinomial}
(a_1+a_2+a_3+\cdots+a_k)^n=n! \sum_{\substack{\lambda\in\mathcal P_{[k]}\\\ell(\lambda)=n}} \frac{a_1^{m_1}a_2^{m_2}a_3^{m_3}...a_k^{m_k}}{m_1!\  m_2!\  m_3!\  ...\  m_k!}.
\end{equation}
If we let $k$ tend to infinity, assuming the infinite sum $a_1+a_2+a_3+\cdots$ converges, the series on the right becomes a sum over all partitions of length $n$. Then dividing both sides of (\ref{Dmultinomial}) by $n!$ and summing over $n\geq 0$, 
the left-hand side yields the Maclaurin series expansion for $\operatorname{exp}(a_1+a_2+a_3+\cdots)$, and the right side can be rewritten as a sum over all partitions: 
\begin{equation}\label{Dmultinomial4}\operatorname{exp}(a_1+a_2+a_3+\cdots)=\sum_{\lambda\in \mathcal P} \frac{a_1^{m_1}a_2^{m_2}a_3^{m_3}...}{m_1!\  m_2!\  m_3!\  ...}.
\end{equation}
To complete the proof, make the substitution $a_k\mapsto a_k x^k$ in (\ref{Dmultinomial4}).
\end{proof}
%

\begin{proof}[Proof of Theorem \ref{Pennthm}]
For $|z|<1, \operatorname{Re}(s)>1$, by \eqref{Maclaurin} we rewrite the product side of generating function \eqref{zetagen} as:
\begin{flalign}
\prod_{n=1}^{\infty}\text{exp}\left(-\log\left(1-{z}{n^{-s}}\right)\right) =\text{exp}
\left(\sum_{n=1}^{\infty}\sum_{j=1}^{\infty}\frac{z^j}{n^{js} j}\right)=\text{exp}
\left(\sum_{j=1}^{\infty}\frac{\zeta(js)}{j }z^j \right).
\end{flalign}
Now setting $x=z$ and $a_j=\zeta(js)/j$ in Lemma \ref{Faa} and noting $N(\lambda)=\prod_{j
\geq 1}j^{m_j}$ for each partition $\lambda$, comparing the co\"efficient of $z^k$ with the 
right-hand side of \eqref{zetagen} gives the theorem.
\end{proof}

Theorem \ref{Pennthm} yields much information about the analytic properties of $\zeta_{\mathcal P}(\{ s\}^k)$.

\begin{corollary}\label{Cor} The partition zeta function $\zeta_{\mathcal P}(\{s\}^k)$:
\begin{enumerate}[i.]
\item inherits analytic continuation from $\zeta(s)$, to the entire complex plane minus poles;
\item has poles at $s=1,\  1/2,\ 1/3,\ 1/4,...,1/k$ with the order of $s=1/j$ being $\left\lfloor k/j \right\rfloor$;
\item has trivial roots at $s=-2, -4, -6, -8,$ etc.;
\item does not have roots at the nontrivial roots of $\zeta(s)$.
\end{enumerate}
\end{corollary}

\begin{proof}
All four items of the corollary follow immediately from well-known analytic properties of the Riemann zeta function. That $\zeta_{\mathcal P}(\{s\}^k)$ inherits analytic continuation from the finite combination of $\zeta(js)$ factors is obvious, and the pole at $s=1/j$ comes from the pole of $\zeta(js)$ at $js=1$ with order of the pole coming from the partition of $k$ with the maximum possible number $\left\lfloor k/j \right\rfloor$ of $j$'s. For $s\in -2\mathbb N$, all the $\zeta(js)$ vanish trivially, giving trivial roots of the partition zeta function; but on the assumption no single value $s$ can serve as a nontrivial root for all the $\zeta(js), 1\leq j \leq k,$ simultaneously, the partition zeta function does not entirely vanish at any nontrivial root of $\zeta(s)$.
\end{proof}

\section{Combinatorial approach}\label{Sect3}

By contrast with the analytic proof of the previous section, in this section we give a combinatorial proof of Theorem \ref{Pennthm}. 
We take an almost identical approach to the second author's work \cite{AS} generating partitions of 
fixed length
and allowing the symmetric group to act on these partitions to create 
\emph{integer compositions}. 

Let $\calC_k$ denote the set of $k$-tuples of positive integers  $c=(c_1, c_2, ..., c_k)$.
These are \emph{integer compositions of length $k$} (or ``$k$-compositions'').
As with partitions, the $c_j$ in $c=(c_1, \dots, c_k)$ are the \emph{parts} of $c$, and the
sum of parts $|c|$ is the \emph{size} of $c$.

We proceed by example. First we recall that $\zeta_{\mathcal P}(\left\{ s \right\}^0)=1$ and the $k=1$ case is simply 
\[ \zeta(s) = \sum_{n\geq 1} n^{-s} = \zeta_\calP (\{s\}^1). \]
The $k=2$ case of Theorem~\ref{Pennthm} may be proved combinatorially as follows:  \vspace{.005in}
\begin{align*}
& \sum_{\lambda\vdash 2} \frac{\zeta(s)^{m_1} \zeta(2s)^{m_2} }{N(\lambda)
m_1! m_2!}
= \frac{1}{2} \zeta(s)^2 + \frac 12 \zeta(2s) \\
&= \frac12 \sum_{n_1\geq 1} n_1^{-s} \sum_{n_2\geq 1} n_2^{-s}
  + \frac 12 \sum_{n\geq 1} (n^2)^{-s} \\
& = \frac 12 \sum_{(n_1,n_2)\in\calC_2} \frac{1}{(n_1n_2)^s}
   + \frac 12\underset{n_1=n_2}{\sum_{(n_1,n_2)\in\calC_2 }}
    \frac{1}{(n_1n_2)^s} \\
& =  \frac 12\left( \underset{n_1\neq n_2}{\sum_{(n_1,n_2)\in\calC_2}} \frac{1}{(n_1n_2)^s}
    +  \underset{n_1=n_2}{\sum_{(n_1,n_2)\in\calC_2} }\frac{1}{(n_1n_2)^s} \right)
   + \frac 12\underset{n_1=n_2}{\sum_{(n_1,n_2)\in\calC_2 }}
    \frac{1}{(n_1n_2)^s} \\   
 & =  \frac 12 \underset{n_1\neq n_2}{\sum_{(n_1,n_2)\in\calC_2}} \frac{1}{(n_1n_2)^s}
    + \left(\frac 12 + \frac 12 \right) \underset{n_1=n_2}{\sum_{(n_1,n_2)\in\calC_2 }}
    \frac{1}{(n_1n_2)^s} \\   
  & =  \frac 12 \underset{n_1\lessgtr n_2}{\sum_{(n_1,n_2)\in\calC_2}} \frac{1}{(n_1n_2)^s}
    + \underset{n_1=n_2}{\sum_{(n_1,n_2)\in\calC_2 }}
    \frac{1}{(n_1n_2)^s} \\        
 &=  \sum_{ n_1 > n_2 \geq 1} \frac{1}{(n_1 n_2)^s} 
     +  \sum_{n_1=n_2 \geq 1} \frac{1}{(n_1n_2)^s }   \\
 &= \sum_{\substack{\pi \in \mathcal P\\ \ell(\pi)=2}}\frac{1}{N(\pi)^{s}}   = \zeta_{\calP} (\{ s \}^2)    .       
\end{align*} 

The details for the case $k=3$ are as follows:

\begin{align*}
&\sum_{\lambda\vdash 3} \frac{\zeta(s)^{m_1} \zeta(2s)^{m_2} \zeta(3s)^{m_3}}
{N(\lambda)m_1! m_2! m_3!} = \frac{1}{6} \Big( \zeta(s) \Big)^3 + \frac 12 \zeta(2s)\zeta(s) + \frac 13 \zeta(3s) \\
&= \frac16 \sum_{n_1\geq 1} n_1^{-s} \sum_{n_2\geq 1} n_2^{-s} \sum_{n_3\geq 1} n_3^{-s}
  + \frac 12 \sum_{n_1\geq 1} (n_1^2)^{-s} \sum_{n_2\geq 1} n_2^{-s}
  + \frac 13 \sum_{n\geq 1} (n^3)^{-s} \\
& = \left[ \frac 16 \sum_{(n_1,n_2,n_3)\in\calC_3} 
   + \left( \frac 16 \underset{n_1=n_2}{ \sum_{(n_1,n_2,n_3)\in\calC_3} }  
   + \frac 16\underset{n_1=n_3}{ \sum_{(n_1,n_2,n_3)\in\calC_3} }  
   + \frac 16\underset{n_2=n_3}{ \sum_{(n_1,n_2,n_3)\in\calC_3} } \right) \right. \\
&  \qquad     \left.            + \frac 13\underset{n_1=n_2=n_3}{\sum_{(n_1,n_2,n_3)\in\calC_3 }}
   \right] \frac{1}{(n_1n_2n_3)^s}  \\
&= \left[  \frac16 \left( \sum_{ \scriptstyle{\text{all $n_i$ unequal} } }
                                +\sum_{ n_1=n_2\neq n_3 } 
                                 +\sum_{ n_1=n_3\neq n_2 } 
                                  +\sum_{ n_1\neq n_2= n_3 } 
                                  +\sum_{ n_1= n_2= n_3 } 
                \right)
 \right. \\ 
  & \qquad +\frac 16 \left( \sum_{ n_1=n_2\neq n_3 }  + \sum_{n_1=n_2=n_3} \right)
                  + \frac 16 \left( \sum_{ n_1 = n_3\neq n_2} + \sum_{n_1=n_2=n_3} \right)\\
  & \qquad     \left.            + \frac 16 \left( \sum_{ n_1 \neq n_2 = n_3} + \sum_{n_1=n_2=n_3} \right)
     + \frac 13 \sum_{ n_1=n_2=n_3 } \right] \frac{1}{(n_1 n_2 n_3)^s}
\\ 
&= \left[  \frac16 \left( \sum_{ \scriptstyle{\text{all $n_i$ unequal} } } \right) 
     +\frac 13 \left( \sum_{ n_1=n_2\neq n_3 } +  \sum_{ n_1=n_3\neq n_2 }  
     +\sum_{ n_1\neq n_2= n_3 }  \right)  
     + 1 \sum_{n_1=n_2=n_3} \right] \frac{1}{(n_1n_2 n_3)^s } \\   
&= \left[  \frac16 \sum_{ \scriptstyle{\text{all $n_i$ unequal} } } 
     +\frac 13   
     \underset{n_1\lessgtr n_2=n_3}{\underset{n_1=n_3\lessgtr n_2,\mathrm{or}}{\sum_{ n_1=n_2\lessgtr n_3, \mathrm{or} }}}   
     +  \sum_{n_1=n_2=n_3} \right] \frac{1}{(n_1n_2 n_3)^s } \\ 
 &= \left[  \sum_{ n_1 > n_2 > n_3 \geq 1  } 
     + \left( \sum_{n_1 > n_2 = n_3 \geq 1} + \sum_{n_1 = n_2 > n_3 \geq 1} \right)  
     +  \sum_{n_1=n_2=n_3} \right] \frac{1}{(n_1n_2 n_3)^s }   \\
 &= \sum_{\substack{\pi \in \mathcal P\\\ell(\pi)=3}}\frac{1}{N(\pi)^{s}}  = \zeta_{\calP} (\{ s \}^3),     
\end{align*} 
where ``$a \lessgtr b$'' means $a\neq b$, but we mean to emphasize that only two of
the six cases where exactly two of the $n_i$ are equal, namely ``$n_1>n_2=n_3$'' 
and ``$n_1=n_2>n_3$," yield a partition. 

As the length $k$ increases, 
analogous combinatorial principles of splitting up, sorting compositions and regrouping are at play in the indices of summation. For the corresponding proof of the general $k$ case, the reader is referred to \cite{AS}; while they are too lengthy to conveniently reproduce here, the exact steps of the proof of MacMahon's partial fraction formula given in \cite{AS} also prove Theorem \ref{Pennthm} under certain substitutions, which we detail in the next section.
\section{Correspondence with MacMahon partial fractions}
The combinatorial identities in Section \ref{Sect3} mimic the second author's proof \cite{AS} of MacMahon's partial fraction decomposition of the generating function for partitions of length 
$\leq k$ 
(see \cite{M1,M2}). 
Let us recall MacMahon's result.

%
%
\begin{MD}
For $|q|<1$, we have 
\begin{equation} \label{MacDecomp}
{\prod_{j=1}^{k}\frac{1}{1-q^j}}=\sum_{\lambda \vdash k}\frac{1}{N(\lambda)\  m_1!\  m_2!\  m_3!\cdots m_k! \  (1-q)^{m_1}(1-q^2)^{m_2}(1-q^3)^{m_3}\cdots(1-q^k)^{m_k}}.
\end{equation}
\end{MD}

Multiplying both sides of~\eqref{MacDecomp} 
by $q^k$ and noting for every partition of $k$ that $k=m_1+2m_2+3m_3+\cdots+km_k$, 
gives the generating function for partitions of length {\it exactly} $k$: 
\begin{equation}\label{MMtimesq}
\frac{q^k}{\prod_{j=1}^{k}(1-q^j)}=\sum_{\lambda \vdash k}\frac{q^{m_1}q^{2m_2}q^{3m_3}\cdots q^{k m_k}}{N(\lambda)\  m_1!\  m_2!\  m_3!\cdots m_k! \  (1-q)^{m_1}(1-q^2)^{m_2}(1-q^3)^{m_3}\cdots(1-q^k)^{m_k}}.
\end{equation}
We point out that there is a 
simple bijection between partitions generated in \eqref{MMtimesq} and 
those in \eqref{MacDecomp}: for a partition $\lambda = (\lambda_1, 
\lambda_2,\dots, \lambda_j)$
of length $j\leq k$, identify $\lambda$ with the $k$-tuple
$(\lambda_1, \lambda_2,\dots, \lambda_j, 0,0,\dots,0)$ and then add $1$ to each 
component of the $k$-tuple to obtain a partition of length exactly $k$.


Then comparing equation \eqref{MMtimesq} to Theorem \ref{Pennthm}, 
one sees an apparent correspondence between $q$-series generating functions 
and the respective zeta function components in Theorem \ref{Pennthm}.\footnote{Similar correspondences between $q$-series and zeta functions have been studied by the first author and his collaborators (see \cite{ORS, OSW, Robert_zeta, PhD}).} We summarize these observations in the 
table below; corresponding entries evidently encode the same partition-theoretic information:
\vspace{0.25cm}

\begin{center}
\begin{tabular}{ | c | c | }
\hline
{\bf Generating function component} & {\bf Analogous zeta function component} \\ \hline
$q^{|\lambda|}$ & $N(\lambda)^{-s}$  \\ \hline
 $\frac{q^j}{1-q^j}$ & $\zeta(js)$  \\ \hline
$\frac{q^k}{\prod_{j=1}^{k}(1-q^j)}$ & $\zeta_{\mathcal P}(\left\{ s \right\}^k)$ \\ \hline
\end{tabular}
\end{center}
\vspace{0.25cm}

The analogy between the left-hand sides of~\eqref{MMtimesq} 
and Theorem \ref{Pennthm} is clear when we compare the following summation representations:
\begin{equation}\label{MMzeta}
\frac{q^k}{\prod_{j=1}^{k}(1-q^j)} = \sum_{\ell(\lambda)=k}q^{|\lambda|}\  \  \  \longleftrightarrow\  \   \  \zeta_{\mathcal P}(\left\{ s \right\}^k)=\sum_{\ell(\lambda)=k}N(\lambda)^{-s}. \end{equation}
There is a one-to-one correspondence between the terms in each summation, i.e., between the partitions being generated. We note that multiplication of terms of either shape $q^n$ or $n^{-s}$ generates partitions in exactly the same way, viz. for  partition $\lambda=(\lambda_1, \lambda_2,\lambda_3,...,\lambda_r)$: 
\begin{equation}\label{prod}
q^{\lambda_1} q^{\lambda_2}q^{\lambda_3} \cdots q^{\lambda_r} = q^{|\lambda|} \  \  \longleftrightarrow \  \  \lambda_1^{-s}{\lambda_2^{-s}} {\lambda_3^{-s}}\cdots \lambda_r^{-s}= N\left( \lambda \right)^{-s}.\end{equation} 
Therefore, taken in finite combinations, geometric series and zeta functions behave identically as partition generating functions\footnote{For infinite combinations, a different analogy holds, between product generating functions like $\prod_{j=1}^{\infty}(1-q^j)^{-1}$ and corresponding Euler products (see \cite{Robert_zeta}).}. Whereas the size $|\lambda|$ is the partition-encoding statistic in the former scheme, the norm $N(\lambda)$ encodes partitions in the latter. 

Moreover, the summand $q^{jn}$ in the geometric series $\sum_{n=1}^{\infty}q^{jn}$ 
and the respective term $n^{-js}$ of $\zeta(js)$ can both be viewed as encoding the partition $(n)^j:=(n,n,...,n)$ consisting of $j$ repetitions of the part $n$, 
viz.  
\begin{equation}\label{geomzeta}
\frac{q^j}{1-q^j} = \sum_{n=1}^{\infty}q^{|(n)^j|}\  \  \  \longleftrightarrow\  \   \  \zeta(js)=\sum_{n=1}^{\infty}N\left((n)^j\right)^{-s}. \end{equation}
The $n$th terms of the summations in \eqref{geomzeta} are in one-to-one correspondence.

Now, since the multiplicities $m_1, m_2, ..., m_k$ associated to any length-$k$ partition must add up to $k$, these nonzero multiplicities themselves represent a permutation of some partition of size $k$. This observation brings compositions into the picture, 
and also provides a link between partitions of length $k$ and those of size $k$. 

To prove MacMahon's partial fraction formula in \cite{AS}, 
roughly speaking, 
one replaces the term $q^j$ 
with a multivariate product $x_1 x_2 x_3 \cdots x_j,\  |x_i|<1$, then permutes the $x_i$'s. More exactly, 
the geometric factors in the right-hand summands of \eqref{MacDecomp} will be rewritten
\begin{equation}\label{MMpart}
\frac{1}{(1-x_1)^{m_1}\  (1-x_2 x_3)^{m_2}\  (1-x_4 x_5 x_6)^{m_3}\cdots}.
\end{equation}
One then permutes the $x_i$'s of nonempty geometric factors, effectively giving rise to compositions, while keeping track of the permutations in the indices of summation as in the examples in Section \ref{Sect3}, using properties of the symmetric group to enumerate multiply-counted terms. 
In the end one sets each dummy variable $x_i=q$ to produce \eqref{MacDecomp}. 

The counting arguments of \cite{AS} are unaltered in the case of partitions of length exactly $k$. One 
can rewrite the geometric aspect of each right-hand summand of \eqref{MMtimesq} as
\begin{equation}\label{MMpart2}
\frac{x_1^{m_1} (x_2 x_3)^{m_2} (x_4 x_5 x_6)^{m_3} \cdots}{(1-x_1)^{m_1}\  (1-x_2 x_3)^{m_2}\  (1-x_4 x_5 x_6)^{m_3}\cdots},
\end{equation}
and by precisely the same steps taken in \cite{AS}, equation \eqref{MMtimesq} is proved. Going a step further in this direction: as $|x_i|<1$, one can rewrite the $j$th geometric factor of \eqref{MMpart2} above in series form: 
\begin{equation}\label{MMhalfway}
\frac{x_{i_1} x_{i_{2}}\cdots x_{i_j}}{1-x_{i_1} x_{i_{2}}\cdots x_{i_j}}=\sum_{n=1}^{\infty} x_{i_1}^n x_{i_{2}}^n\cdots x_{i_j}^n,\end{equation}
noting $i_{m+1}=i_m+1$ in the indices, 
then in the sum on the right side of \eqref{MMhalfway} make the change
\begin{equation}
x_i^{n}\mapsto n^{-x_i},
\end{equation}
with the $x_i$'s now needing to satisfy $\operatorname{Re}(x_{i_1}+x_{i_2}+\cdots +x_{i_j})>1$ for convergence of the series. This mapping produces a one-to-one term-wise correspondence:
\begin{equation}\label{above}
\sum_{n=1}^{\infty}x_{i_1}^n x_{i_{2}}^n\cdots x_{i_j}^n \  \  \  \longleftrightarrow\  \   \    \sum_{n=1}^{\infty}n^{-x_{i_1}} n^{-x_{i_2}}\cdots n^{-x_{i_j}}.
\end{equation}
 Moreover, the terms of both these series are in one-to-one correspondence with the summands of the $j$th geometric factor of \eqref{MMpart} when it is expanded as a series:
\begin{equation}\label{MMhalfway2}
\frac{1}{1-x_{i_1} x_{i_{2}}\cdots x_{i_j}}=\sum_{n=1}^{\infty} x_{i_1}^{n-1} x_{i_{2}}
^{n-1}\cdots x_{i_j}^{n-1},
\end{equation}
and 
the symmetric group acts on the $x_i$'s identically in finite combinations of series of any one of these types.

Then one can replace the $j$th geometric factor 
of \eqref{MMpart}, $j=1,2,3,...,k$, with the right-hand multivariate zeta series in \eqref{above}, permute the $x_{i}$'s and enumerate permutations following exactly the steps in \cite{AS}; finally, 
one sets each $x_{i}=s$ such that  $x_{i_1}+x_{i_2}+\cdots+x_{i_j}=js$, 
to arrive at Theorem \ref{Pennthm}.

 \end{document}